\documentclass[12pt]{article}
\usepackage{amsmath}
\usepackage{float}
\usepackage{amsthm}
\usepackage{amssymb}
\usepackage[utf8x]{inputenc}
\usepackage{caption}
\usepackage{graphics}
\usepackage{enumerate} 
\usepackage{algorithmic}
\usepackage[Algorithm,ruled]{algorithm}
\usepackage{titlesec}
\usepackage{sectsty}
\titleformat*{\section}{\LARGE\bfseries}
\titleformat*{\subsection}{\Large\bfseries}
\titleformat*{\subsubsection}{\large\bfseries}
\sectionfont{\Huge}
\subsectionfont{\large}
\subsubsectionfont{\normalsize}
\usepackage{multicol}
\usepackage{lmodern}
\usepackage{marvosym} 
\usepackage{lipsum}
\usepackage{mwe}
\usepackage{caption}
\usepackage{subcaption} 
\newtheoremstyle{case}{}{}{}{}{}{:}{ }{}
\theoremstyle{case}

\newcommand{\be}{\begin{equation}}
\newcommand{\ee}{\end{equation}}
\newcommand{\ben}{\begin{eqnarray*}}
\newcommand{\een}{\end{eqnarray*}}
\newtheorem{examp}{\sc example}
\newtheorem{remk}{\sc remark}
\newtheorem{corol}{\sc corollary}
\newtheorem{lemma}{\sc lemma}
\newtheorem{theorem}{\sc theorem}
\newtheorem{defn}{\sc definition}
\newcommand{\bt}{\begin{theorem}}
\newcommand{\et}{\end{theorem}}
\newcommand{\bl}{\begin{lemma}}
\newcommand{\el}{\end{lemma}}
\newcommand{\bed}{\begin{defn}}
\newcommand{\eed}{\end{defn}}
\newcommand{\brem}{\begin{remk}}
\newcommand{\erem}{\end{remk}}
\newcommand{\bex}{\begin{examp}}
\newcommand{\eex}{\end{examp}}
\newcommand{\bcl}{\begin{corol}}
\newcommand{\ecl}{\end{corol}}

\newcommand{\NI}{\noindent}

\theoremstyle{definition}
\theoremstyle{remark}

\numberwithin{equation}{section}
\numberwithin{theorem}{section}
\numberwithin{lemma}{section}
\newtheorem{definition}{Definition}[section]

\begin{document}

\title{\large\bf\sc Computation of a new error bound for tensor complementarity problem with P tensor }

\author{R. Deb$^{a,1}$, A. Dutta$^{a,2}$, A. K. Das$^{b,3}$\\
\emph{\small $^{a}$Jadavpur University, Kolkata , 700 032, India.}\\	
\emph{\small $^{b}$Indian Statistical Institute, 203 B. T.
	Road, Kolkata, 700 108, India.}\\
\emph{\small $^{1}$Email: rony.knc.ju@gmail.com}\\
\emph{\small $^{2}$Email: aritradutta001@gmail.com}\\
\emph{\small $^{3}$Email: akdas@isical.ac.in}\\
}

\date{}

\maketitle

\begin{abstract}
\NI We propose a new error bound for the solution of tensor complementarity problem TCP$(q, \mathcal{A})$ given that $\mathcal{A}$ is a $P$-tensor and $q$ is a real vector. We show that the proposed error bound is sharper than the earlier version of error bound available in the literature.  We establish absolute and relative error bound for TCP$(q, \mathcal{A})$ where $\mathcal{A}$ is an even order positive diagonal tensor.\\

\\

\end{abstract}
\footnotetext[1]{Corresponding author}

\section{Introduction}
In this article a new error bound sharper than the earlier version is introduced for the tensor complementarity problem with $P$-tensor. During last several years, the tensor complementarity problem attains much attraction and has been studied extensively with respect to theory, to solution methods and applications. In recent years, various tensors with special structures have been studied. For details, see \cite{qi2017tensor} and \cite{song2015properties}. The tensor complementarity problem was studied initially by Song and Qi \cite{song2014properties}.
The tensor complementarity problem is a subclass of the non-linear complementarity problems where the function involved in the non-linear complementarity problem is a special polynomial defined by a tensor in the tensor complementarity problem. The polynomial functions used in tensor complementarity problems have some special structures.
\\
 For a given mapping $F: \mathbb{R}^n \mapsto \mathbb{R}^n$ the complementarity problem is to find a vector $x\in \mathbb{R}^n $ such that
\begin{equation}\label{Classical comp problem}
   x\geq 0, ~~F(x) \geq 0, ~~ \mbox{and}~ x^{T}F(x)=0.
\end{equation}
If $F$ is nonlinear mapping, then the problem (\ref{Classical comp problem}) is called a nonlinear complementarity problem \cite{facchinei2007finite}, and if $F$ is linear function, then the problem (\ref{Classical comp problem}) reduces to a linear complementarity problem \cite{cottle2009linear}. The linear complementarity problem may be defined as follows:\\
Given a matrix $ M\in \mathbb{R}^{n\times n} $ and a vector $q \in \mathbb{R}^n$, the linear complementarity problem \cite{cottle2009linear}, denoted by $LCP(q,M)$, is to find a pair of vectors $w,z \in\mathbb{R}^n$ such that 
\begin{equation}\label{linear complementarity problem}
    z\geq 0, ~~~ w=Mz+q\geq 0, ~~~ z^T w=0.
\end{equation}

\noindent The algorithm based on principal pivot transforms namely Lemke's algorithm, Criss-cross algorithm which are used to find the solutions of linear complementarity problem are studied extensively considering several matrix classes. 

It is important that large number of formulations not only enrich the linear complementarity problem but also generate different matrix classes along with their computational methods.
 For details see \cite{das2019some}, \cite{neogy2006some}, \cite{dutta2021some} \cite{dutta2022on}, \cite{neogy2013weak}, \cite{jana2018semimonotone}, \cite{neogy2005almost}, \cite{neogy2011singular}, \cite{jana2019hidden}, \cite{jana2021more}, \cite{neogy2009modeling}, \cite{das2017finiteness}, \cite{das2018invex}. For details of game theory see \cite{mondal2016discounted}, \cite{neogy2008mathematical}, \cite{neogy2008mixture}, \cite{neogy2005linear}, \cite{neogy2016optimization}, \cite{das2016generalized}, \cite{neogy2011generalized} and for details of QMOP see \cite{mohan2004note}. Even matrix classes arise during the study of Lemke's algorithm as well as principal pivot transform. For details see \cite{mohan2001classes}, \cite{mohan2001more} \cite{neogy2005principal}, \cite{das2016properties}, \cite{neogy2012generalized}, \cite{jana2019hidden}, \cite{jana2021iterative}, \cite{jana2021more}, \cite{jana2018processability}. 
Now we consider the case of $F(x)=\mathcal{A}x^{m-1} + q$ with $\mathcal{A}\in T_{m,n}$ and $ q\in \mathbb{R}^n$ then the problem (\ref{Classical comp problem}) becomes
\begin{equation}\label{ Tensor Complementarity problem}
    x\geq 0, ~~~\mathcal{A}x^{m-1} + q\geq 0, ~~~\mbox{and}~~ x^{T}(\mathcal{A}x^{m-1}+q)=0
\end{equation}
which is called a tensor complementarity problem, denoted by the TCP$(q,\mathcal{A})$.
Denote $\omega=\mathcal{A}x^{m-1} + q$, then the tensor complementarity problem is to find $x$ such that 
\begin{equation}\label{ Tensor Complementarity problem 2}
    x\geq 0, ~~~ \omega= \mathcal{A}x^{m-1} + q\geq 0, ~~~\mbox{and}~~ x^{T}\omega=0.
\end{equation}

\noindent Motivated by the discussion on positive definiteness of multivariate homogeneous polynomial forms \cite{bose1976general}, \cite{hasan1996procedure}, \cite{jury1981positivity}, Qi \cite{qi2005eigenvalues} introduced the concept of symmetric positive definite (positive semi-definite) tensors. Song and Qi \cite{song2015properties} studied $P(P_0)$-tensors and $B(B_0)$-tensors.
The equivalence between (strictly) semi-positive tensors and (strictly) copositive tensors in symmetric case   were shown by Song and Qi \cite{song2016tensor}. The existence and uniqueness of solution of TCP$(q,\mathcal{A})$ with some special tensors were discussed by Che, Qi, Wei \cite{che2016positive}. The boundedness of the solution set of the TCP$(q,\mathcal{A})$ was studied by Song and Yu \cite{song2016properties}. The sparse solutions to TCP$(q,\mathcal{A})$ with a $Z$-tensor and its method to calculate were obtained by Luo, Qi and Xiu \cite{luo2017sparsest}. The equivalent conditions of solution to TCP$(q, \mathcal{A})$ were shown by Gowda, Luo, Qi and Xiu \cite{gowda2015z} for a $Z$-tensor $\mathcal{A}$. The global uniqueness of solution of TCP$(q, \mathcal{A})$ was considered by Bai, Huang and Wang \cite{bai2016global} for a strong $P$-tensor $\mathcal{A}$.
The properties of TCP$(q, \mathcal{A})$ was studied by Ding, Luo and Qi \cite{ding2018p} for a new class of $P$-tensor. The properties of the several classes of $Q$-tensors were presented by Suo and Wang \cite{huang2015q}. In this article we introduce column adequate tensor in the context of tensor complementarity problem and study different properties of this tensor.\\
The paper is organised as follows. Section 2 contains some basic notations and results. In Section 3, We propose a new error bound for the solution of tensor complementarity problem TCP$(q, \mathcal{A})$ given that $\mathcal{A}$ is a $P$-tensor and $q$ is a real vector. We show that the proposed error bound is sharper than the earlier version of error bound available in the literature.  We establish absolute and relative error bound for TCP$(q, \mathcal{A})$ where $\mathcal{A}$ is an even order positive diagonal tensor. The results are illustrated with the help of a numerical example.

\section{Preliminaries}

We begin by introducing some basic notations used in this paper. We consider tensor, matrices and vectors with real entries. Let $m$th order $n$ dimensional real tensor $\mathcal{A}= (a_{i_1 i_2 ... i_m}) $ be a multidimensional array of entries $a_{i_1 i_2 ... i_m} \in \mathbb{R}$ where $i_j \in [n]$ with $j\in [m]$. $T_{m,n}$ denotes the set of real tensors of order $m$ and dimension $n.$ 
For any positive integer $n,$  $[n]$ denotes set of $\{ 1,2,...,n \}$. All vectors are column vectors. Let $\mathbb{R}^n$ denote the $n$-dimensional Euclidean space and $\mathbb{R}^n_+ :=\{ x\in \mathbb{R}^n : x\geq 0 \}$. For any $x \in \mathbb{R}^n$, let $x^{[m]}\in \mathbb{R}^n$ with its $i$th component being $x^m_i$ for all $i \in [n]$. $||x||_{\infty}= \max \{ |x_i|: i\in [n]\}$ and $||x||_2= \left( \sum_{i =1}^n |x_i|^2 \right)^{\frac{1}{2}}.$
Then for a continuous, positively homogeneous operator $T:\mathbb{R}^n\mapsto \mathbb{R}^n$ it is obvious that $||T||_{\infty}= \max_{||x||_{\infty}=1} ||T(x)||_{\infty}$ is an operator norm of $T$ and $||T(x)||_{\infty} \leq ||T||_{\infty} ||x||_{\infty} $ for any $x\in \mathbb{R}^n.$
For $\mathcal{A}\in T_{m,n} $ and $x\in \mathbb{R},\; \mathcal{A}x^{m-1}\in \mathbb{R}^n $ is a vector defined by
\[ (\mathcal{A}x^{m-1})_i = \sum_{i_2, i_3, ...i_m =1}^{n} a_{i i_2 i_3 ...i_m} x_{i_2} x_{i_3} \cdot\cdot \cdot x_{i_m} , ~~~\mbox{for all}~i \in [n] \]
and $\mathcal{A}x^m\in \mathbb{R} $ is a scalar defined by
\[ \mathcal{A}x^m = \sum_{i_1,i_2, i_3, ...i_m =1}^{n} a_{i_1 i_2 i_3 ...i_m} x_{i_1} x_{i_2} \cdot\cdot \cdot x_{i_m}. \]

\NI For any $\mathcal{A}\in T_{m,n},$ 
\begin{equation*}
    || \mathcal{A} ||_{\infty}= \max_{i\in [n]} \sum_{i_2, ..., i_m =1}^n |a_{i i_2, ..., i_m}| .
\end{equation*}
Song et al. \cite{song2015properties} defined two operators. Let $\mathcal{A}\in T_{m,n}.$ For $x\in \mathbb{R}^n,$ the operator $T_{\mathcal{A}} : \mathbb{R}^n \mapsto \mathbb{R}^n$ is defined by 
\begin{equation}
    T_{\mathcal{A}} x = \left\{\begin{array}{rr}
        ||x||_2^{2-m} \mathcal{A} x^{m-1}, & x \neq 0, \\
        0, & x=0.
    \end{array} \;
    \right.
\end{equation}
When $m$ is even, for $x\in \mathbb{R}^n$ another operator $F_{\mathcal{A}}: \mathbb{R}^n \mapsto \mathbb{R}^n$ is defined by 
\begin{equation}
    F_{\mathcal{A}} x = (\mathcal{A} x^{m-1})^{[\frac{1}{m-1}]}.
\end{equation}
Song et al. \cite{song2015properties} introduced two important quantities 
\begin{equation}
    \alpha(T_{\mathcal{A}})= \min_{||x||_{\infty}=1} \max_{i\in [n]} x_i (T_{\mathcal{A}} x)_i
\end{equation}
 for any positive integer $m$ and 
\begin{equation}\label{defn of alphaFA}
    \alpha(F_{\mathcal{A}})= \min_{||x||_{\infty}=1} \max_{i\in [n]} x_i (F_{\mathcal{A}} x)_i
\end{equation} 
when $m$ is even.
\NI Since $P$-tensors are all even order, then $\alpha(T_{\mathcal{A}})$ and $\alpha(F_{\mathcal{A}})$ are both well defined for any $P$-tensor.
The following result is necessary and sufficient conditions for a P-tensor in
terms of $\alpha(T_{\mathcal{A}})$ and $\alpha(F_{\mathcal{A}}).$

\begin{lemma}\cite{song2015properties}
Let $\mathcal{A}\in T_{m,n}.$ Then\\
(a) $\mathcal{A}$ is a $P$-tensor if and only if $\alpha(T_{\mathcal{A}})>0,$\\
(b) When $m$ is even $\mathcal{A}$ is a $P$-tensor if and only if $\alpha(F_{\mathcal{A}})>0.$
\end{lemma}
\NI Since every strong $P$-tensor is a $P$-tensor, the following corollary is obvious.
\begin{corol}\cite{zheng2019global}
Let $\mathcal{A}\in T_{m,n}.$ Then\\
(a) $\alpha(T_{\mathcal{A}})>0$ if $\mathcal{A}$ is a strong $P$-tensor,\\
(b) $\alpha(F_{\mathcal{A}})>0$ if $\mathcal{A}$ is a strong $P$-tensor.
\end{corol}

\begin{definition}
Let $e : \mathbb{R}^n\mapsto \mathbb{R}^n_+$ be a function. Assume that TCP$(q,\mathcal{A})$ has a nonempty solution set.

\NI (a) We say that $e(x)$ is a residual function of TCP$(q,\mathcal{A})$, if $e(x) \geq 0$, and $e(x) = 0$ if and only if $x$ solves TCP$(q,\mathcal{A})$.

\NI (b) We say that a residual function $e(x)$ is a lower global error bound for TCP$(q,\mathcal{A})$, if $ \exists $ some constants $c_1 > 0$ such that for each $x\in \mathbb{R}^n$ and any solution $\bar{x},$ $c_1 e(x) \leq ||x- \bar{x}||.$

\NI (c) We say that a residual $e(x)$ is an upper global error bound for the TCP$(q,\mathcal{A})$ if there exists some constant $c_2 > 0$ such that for each $x\in \mathbb{R}^n$ and any solution $\bar{x},$ $ ||x-\bar{x}|| \leq c_2 e(x),$ 
\end{definition}

\begin{defn}\cite{song2016properties}
Given $\mathcal{A}= (a_{i_1 i_2 ... i_m}) \in T_{m,n} $ and $q\in \mathbb{R}^n$, a vector $x$ is said to be (strictly) feasible, if $x(>) \geq 0$ and $\mathcal{A}x^{m-1}+q (>) \geq 0$.

\noindent TCP$(q,\mathcal{A})$, defined by equation (\ref{ Tensor Complementarity problem}) is said to be (strictly) feasible if a (strictly) feasible vector exists. 

\noindent TCP$(q,\mathcal{A})$ is said to be solvable if there is a feasible vector $x$ satisfying $x^{T}(\mathcal{A}x^{m-1}+q)=0$ and $x$ is the solution.
\end{defn}

\begin{defn}
\cite{song2015properties} A tensor $\mathcal{A}= (a_{i_1 i_2 ... i_m}) \in T_{m,n} $ is said to be a $P$-tensor, if for each $x\in \mathbb{R}^n \backslash \{0\}$, there exists an index $i\in [n]$ such that $x_i \neq 0$ and $x_i (\mathcal{A}x^{m-1})_i > 0$.
\end{defn}

\begin{defn}\cite{bai2016global}
A tensor $\mathcal{A}\in T_{m,n}$ is said to be strong $P$-tensor if for any two different $x=(x_i)$ and $y=(y_i)$ in $\mathbb{R}^n$, $\max_{i\in [n]} (x_i - y_i)(\mathcal{A}x^{m-1} - \mathcal{A}y^{m-1} )_i >0$.
\end{defn}


\begin{theorem}\cite{bai2016global}
For any $q\in \mathbb{R}^n$ and a $P$-tensor $\mathcal{A} \in T_{m,n}$, the solution set of TCP$(q,\mathcal{A})$ is nonempty and compact.
\end{theorem}

\begin{theorem}\cite{bai2016global}
Let $\mathcal{A}\in T_{m,n}$ be a strong $P$-tensor. Then the TCP$(q, \mathcal{A})$ has the property of global uniqueness and solvability property.
\end{theorem}

\begin{theorem}\cite{song2015error}\label{for solution bound 1}
Let $\mathcal{A}\in T_{m,n} \; (m \geq 0)$ be a $P$-tensor, and let $x$ be a solution of TCP$(q, \mathcal{A}).$ If $m$ be even, then
\begin{equation}
    \frac{||(-q)_+||_{\infty}^{\frac{1}{m-1}}}{||F_{\mathcal{A}}||_{\infty}} \leq ||x||_{\infty} \leq \frac{||(-q)_+||_{\infty}^{\frac{1}{m-1}}}{\alpha(F_{\mathcal{A}})}.
\end{equation}
\end{theorem}

\begin{theorem}\label{first proof of error bound}\cite{zheng2019global}
Given $q\in \mathbb{R}^n,\; \mathcal{A}\in T_{m,n}$ with $\mathcal{A}$ being a $P$-tensor and $\alpha(F_{\mathcal{A}})$ is defined by (\ref{defn of alphaFA}). For any $u\in \mathbb{R}^n,$ let $x$ be a solution of TCP$(q,\mathcal{A}).$ Suppose the residue function $\tilde{v}$ is defined as $\tilde{v}=\min \{ u, [(\mathcal{A}(u-z)^{m-1})^{[\frac{1}{m-1}]} + (\mathcal{A}z^{m-1} + q)^{[\frac{1}{m-1}]} ]\}.$ Then for any $u\in\mathbb{R}^n,$
\begin{equation}\label{first equation of error bound}
    \frac{1}{1+ ||\mathcal{A}||^{\frac{1}{m-1}}}_{\infty} ||\tilde{v}||_{\infty} \leq ||u-x||_{\infty} \leq \frac{1+||\mathcal{A}||^{\frac{1}{m-1}}_{\infty}}{\alpha(F_\mathcal{A})} ||\tilde{v}||_{\infty}
\end{equation}
\end{theorem}

\section{Main results}
We begin with a theorem which provides a solution bound for TCP$(q,\mathcal{A})$ involving $P$-tensor.

\begin{theorem}\label{for solution bound 2}
Let $\mathcal{A}\in T_{m,n}$ be a $P$-tensor where $m$ is even and $x$ be a solution of TCP$(q,\mathcal{A}).$ Then 
\begin{equation}\label{equation for solution bound}
    \frac{||(-q)_+||_{\infty}^{\frac{1}{m-1}}}{||\mathcal{A}||_{\infty}^{\frac{1}{m-1}}} 
    \leq ||x||_{\infty} \leq 
    \frac{||(-q)_+||_{\infty}^{\frac{1}{m-1}}}{\alpha(F_{\mathcal{A}})}.
\end{equation}
\end{theorem}
\begin{proof}
Note that from Theorem \ref{for solution bound 1} we obtain the right hand inequality. We prove the left hand inequality. Since $x$ is a solution of TCP$(q,\mathcal{A})$ we have $\mathcal{A}x^{m-1} + q \geq 0.$ Which implies 
\begin{equation}\label{for bound of solution}
    \mathcal{A}x^{m-1} \geq -q.
\end{equation}
The following inequation is proved by Zheng et al \cite{zheng2019global} in Theorem 3.2. However for the sake of completeness we give the details.
We have 
\begin{align*}
    ||\mathcal{A}x||_{\infty}^{\frac{1}{m-1}}&= \left( \max_{i\in [n]} \left\{ |(\mathcal{A}x^{m-1})_i| \right\} \right)^{\frac{1}{m-1}}\\
    &= \left( \max_{i\in [n]} \left\{ |\sum_{i_2, ..., i_m=1}^n a_{i i_2 ... i_m} x_{i_1} x_{i_2} ... x_{i_m}| \right\} \right)^{\frac{1}{m-1}}\\
    & \leq \left( \max_{i\in [n]} \left\{ \sum_{i_2, ..., i_m=1}^n |a_{i i_2 ... i_m}| \right\} ||x||_{\infty}^{m-1} \right)^{\frac{1}{m-1}}\\
    &= \left( ||\mathcal{A}||_{\infty} ||x||_{\infty}^{m-1} \right)^{\frac{1}{m-1}}\\
    &= ||\mathcal{A}||_{\infty}^{\frac{1}{m-1}} ||x||_{\infty} .
\end{align*}
Thus 
\begin{equation}\label{imp inequation in error bound}
     ||\mathcal{A}x||_{\infty}^{\frac{1}{m-1}} \leq ||\mathcal{A}||_{\infty}^{\frac{1}{m-1}} ||x||_{\infty} .
\end{equation}

\NI Therefore,
\begin{equation*}
    ||\mathcal{A}||_{\infty}^{\frac{1}{m-1}} ||x||_{\infty} \geq ||\mathcal{A}x||_{\infty}^{\frac{1}{m-1}} \geq ||(\mathcal{A}x)_+||_{\infty}^{\frac{1}{m-1}} \geq ||(-q)_+||_{\infty}^{\frac{1}{m-1}} \mbox{ using Equation \ref{for bound of solution}}.
\end{equation*}
Hence the left hand inequality.
\end{proof}

Here we investigate the global error bound for the TCP$(q,\mathcal{A}).$
\begin{theorem}\label{theorem main result}
Let $\mathcal{A} \in T_{m,n} $ be a $P$-tensor, $z\in \mathbb{R}^n$ be a solution of TCP$(q,\mathcal{A})$ and $u$ be an arbitrary vector in $\mathbb{R}^n$. Then 
\begin{equation}\label{equation of main result}
    \frac{||v||_\infty (1 + ||\mathcal{A}||_{\infty}^{\frac{1}{m-1}}) - \sqrt{D}}{2\alpha (F_{\mathcal{A}})} \leq ||z-u||_\infty \leq
    \frac{||v||_\infty (1 + ||\mathcal{A}||_{\infty}^{\frac{1}{m-1}}) + \sqrt{D}}{2\alpha (F_{\mathcal{A}})},\; \forall\; u\in\mathbb{R}^n, 
\end{equation}
where 
\begin{center}
   $v=u- \max \{ 0, u-[(\mathcal{A}(u-z)^{m-1})^{[\frac{1}{m-1}]} + (\mathcal{A}z^{m-1} + q)^{[\frac{1}{m-1}]} ]\},$
\end{center}
and
\begin{center}
   $D = ||v||_{\infty}^2(1+ ||\mathcal{A}||_{\infty}^{\frac{1}{m-1}} )^2 -4\alpha(F_{\mathcal{A}}) v_t^2  \geq 0,$
\end{center}
the quantity $\alpha(F_{\mathcal{A}})$ is defined by (\ref{defn of alphaFA}) and $t$ satisfies $v_t \neq 0 ,$
\begin{equation*}
     \max_{i\in [n]} \{(u-z)_i (\mathcal{A}(u-z)^{m-1})_i \} = (u-z)_t (\mathcal{A}(u-z)^{m-1})_t.
\end{equation*}

\end{theorem}
\begin{proof}
Consider the TCP$(q,\mathcal{A})$, which is to find $z\in \mathbb{R}^n$ such that,
\begin{equation}
  z \geq 0, \;\; \mathcal{A}z^{m-1} +q \geq 0, \;\; z^T(\mathcal{A}z^{m-1} +q) =0.
\end{equation}
Let the solution set of TCP$(q,\mathcal{A})$ be denoted by $SOL(\mathcal{A},q)= \{ z\in \mathbb{R}^n : \; z \geq 0, \; \mathcal{A}z^{m-1} +q \geq 0, \; z^T(\mathcal{A}z^{m-1} +q) =0 \}$. Let $z \in SOL(\mathcal{A},q).$ Then for $w=(\mathcal{A}z^{m-1} +q)^{[\frac{1}{m-1}]} \geq 0,$ we have
\begin{equation}
    z \geq 0, \;\;\;\; w \geq 0, \;\;\;\; z^Tw= 0.
\end{equation}
Let $v=u- \max \{ 0, u-[(\mathcal{A}(u-z)^{m-1})^{[\frac{1}{m-1}]} + (\mathcal{A}z^{m-1} + q)^{[\frac{1}{m-1}]} ]\}.$ Now we consider the vector $y= u-v$ and $x=  (\mathcal{A}(u-z)^{m-1})^{[\frac{1}{m-1}]} + (\mathcal{A}z^{m-1} + q)^{[\frac{1}{m-1}]} -v $.
Then
\begin{center}
   $y= u-v =\max \{ 0, u-[(\mathcal{A}(u-z)^{m-1})^{[\frac{1}{m-1}]} + (\mathcal{A}z^{m-1} + q)^{[\frac{1}{m-1}]} ]\} \geq 0.$ 
\end{center}
Again,
\begin{align*}
    x &=(\mathcal{A}(u-z)^{m-1})^{[\frac{1}{m-1}]} + (\mathcal{A}z^{m-1} + q)^{[\frac{1}{m-1}]}) -v\\
      &=(\mathcal{A}(u-z)^{m-1})^{[\frac{1}{m-1}]} + (\mathcal{A}z^{m-1} + q)^{[\frac{1}{m-1}]} -u\\
      & \quad \quad+\max \{ 0, u-[(\mathcal{A}(u-z)^{m-1})^{[\frac{1}{m-1}]} + (\mathcal{A}z^{m-1} + q)^{[\frac{1}{m-1}]} ]\}\\
      &=\max \{ 0, [(\mathcal{A}(u-z)^{m-1})^{[\frac{1}{m-1}]} + (\mathcal{A}z^{m-1} + q)^{[\frac{1}{m-1}]} ]-u \} \geq 0.
\end{align*}

\NI Also by the construction of the vectors $y$ and $x$ we have $y_i x_i =0,\;\forall\; i\in [n.]$
Thus the vectors $y$ and $x$ satisfy the following inequalities and complementarity condition. 
\begin{equation}
    y \geq 0, \;\;\;\; x \geq 0, \;\;\;\; y^Tx= 0.
\end{equation}
Then for $i\in [n],$
\begin{equation}\label{error bound proof 1}
    (y-z)_i(x-w)_i = y_i x_i - z_i x_i - y_i w_i + z_i w_i = -z_ix_i -y_iw_i \leq 0.
\end{equation}
Again
\begin{align*}
    (y-z)_i(x-w)_i & = (u-v-z)_i((\mathcal{A}(u-z)^{m-1})^{[\frac{1}{m-1}]} -v )_i \\
      & = (u-z)_i((\mathcal{A}(u-z)^{m-1})^{[\frac{1}{m-1}]})_i -v_i((\mathcal{A}(u-z)^{m-1})^{[\frac{1}{m-1}]})_i \\
       & -(u-z)_iv_i + v^2_i \\
      & \leq  0 \;\;\mbox{ by (\ref{error bound proof 1}).} 
\end{align*}
Hence for each $i$ we obtain
\begin{equation}\label{connecting equation}
    (u-z)_i ((\mathcal{A}(u-z)^{m-1})^{[\frac{1}{m-1}]})_i \leq v_i((\mathcal{A}(u-z)^{m-1})^{[\frac{1}{m-1}]})_i + (u-z)_iv_i - v^2_i .
\end{equation}
As $\mathcal{A}$ is a $P$-tensor, $\max_{i\in [n]} \{ x_i(\mathcal{A}x^{m-1})_i \} > 0, \; \forall \; x \; \in \mathbb{R}^n\backslash \{0\}  $.
Let $t$ be the particular index for which 
\begin{equation}\label{crutial inequality of error bound}
    \max_{i \in [n]} \{ (u-z)_i ((\mathcal{A}(u-z)^{m-1})^{[\frac{1}{m-1}]})_i\} = (u-z)_t ((\mathcal{A}(u-z)^{m-1})^{[\frac{1}{m-1}]})_t \geq 0 .
\end{equation}

\NI Now using from (\ref{imp inequation in error bound}) we have
\begin{equation}\label{comparison of norm}
    ||\mathcal{A}(u-z)^{m-1}||_\infty^{\frac{1}{m-1}} \leq ||\mathcal{A}||_\infty^{\frac{1}{m-1}} ||u-z||_\infty .
\end{equation}

\NI Again from Equation 14 of \cite{zheng2019global} we have,
\begin{equation}\label{comparison of norm with alphaF}
    \alpha(F_{\mathcal{A}})||z-u||_\infty^2 \leq \max_{i \in [n]} \{ (u-z)_i (\mathcal{A}(u-z)^{m-1})^{[\frac{1}{m-1}]}_i\}.
\end{equation}

\NI Then from (\ref{connecting equation}) (\ref{crutial inequality of error bound}), (\ref{comparison of norm}) and (\ref{comparison of norm with alphaF}) we obtain, 
\begin{align*}
   \alpha(F_{\mathcal{A}})||z-u||_\infty^2 & \leq \max_{i \in [n]} \{ (u-z)_i (\mathcal{A}(u-z)^{m-1})^{[\frac{1}{m-1}]}_i\}\\
     & \leq (u-z)_tv_t + v_t((\mathcal{A}(u-z)^{m-1})^{[\frac{1}{m-1}]})_t - v^2_t\\
     & \leq ||u-z||_{\infty} ||v||_{\infty} + ||\mathcal{A}(u-z)^{m-1}||_\infty^{\frac{1}{m-1}} ||v||_{\infty} - v_t^2 \\
     & \leq ||v||_{\infty} ||u-z||_{\infty} + ||\mathcal{A}||_\infty^{\frac{1}{m-1}} ||u-z||_\infty ||v||_{\infty} - v_t^2 \\
     & = (1+ ||\mathcal{A}||_{\infty}^{\frac{1}{m-1}}) ||v||_{\infty} ||u-z||_\infty - v_t^2 .
\end{align*}

\NI Therefore, 
\begin{equation}\label{main equation}
    \alpha(F_{\mathcal{A}})||z-u||_\infty^2 \leq (1+ ||\mathcal{A}||_{\infty}^{\frac{1}{m-1}}) ||v||_{\infty} ||u-z||_\infty - v_t^2
\end{equation}
from (\ref{main equation}) we observe that if $v_t=0$ then $u=z$, which means $u$ is the true solution. Thus if $v_t \neq 0$, then from (\ref{main equation}) we obtain,
\begin{equation}\label{main equation revised}
    \alpha(F_{\mathcal{A}})||z-u||_\infty^2  -(1+ ||\mathcal{A}||_{\infty}^{\frac{1}{m-1}}) ||v||_{\infty} ||u-z||_\infty + v_t^2 \leq 0 .
\end{equation}
Solving the above inequality we obtain the desired relation.
\end{proof}

Here we propose a relative error bound for the approximate solution for TCP$(q,\mathcal{A})$ provided that the involved tensor $\mathcal{A}$ is a $P$-tensor.

\begin{theorem}
Let $\mathcal{A} \in T_{m,n} $ be a $P$-tensor and $0\neq z\in \mathbb{R}^n$ be a solution of TCP$(q,\mathcal{A})$ and $u$ be an arbitrary vector in $\mathbb{R}^n$. If $(-q)_+ \neq 0,$ we have
\begin{equation}\label{equation of realtive error 1}
    \frac{||v||_\infty (1 + ||\mathcal{A}||_{\infty}^{\frac{1}{m-1}}) - \sqrt{D}}{2 ||(-q)_+||_{\infty}^{\frac{1}{m-1}}}
    \leq \frac{||z-u||_\infty}{||z||_\infty} \leq
    \frac{||\mathcal{A}||_{\infty}^{\frac{1}{m-1}} [||v||_\infty (1 + ||\mathcal{A}||_{\infty}^{\frac{1}{m-1}}) + \sqrt{D}]}{2\alpha (F_{\mathcal{A}}) ||(-q)_+||_{\infty}^{\frac{1}{m-1}}}
\end{equation}
where $v=u- \max \{ 0, u-[(\mathcal{A}(u-z)^{m-1})^{[\frac{1}{m-1}]} + (\mathcal{A}z^{m-1} + q)^{[\frac{1}{m-1}]} ]\}$,

\noindent$D = ||v||_{\infty}^2(1+ ||\mathcal{A}||_{\infty}^{\frac{1}{m-1}} )^2 -4\alpha(F_{\mathcal{A}}) v_t^2  \geq 0$ and $t$ satisfies $v_t \neq 0,$
\begin{equation*}
    (u-z)_t (\mathcal{A}(u-z)^{m-1})_t =\max_{i\in [n]} \{(u-z)_i (\mathcal{A}(u-z)^{m-1})_i \}.
\end{equation*}
\end{theorem}

\begin{proof}
From Theorem \ref{for solution bound 2} we have inequation (\ref{equation for solution bound}). Since $(-q)_+ \neq 0,$ and $z\neq 0$ the inequation (\ref{equation for solution bound}) gives
\begin{equation}\label{equation for inverse of solution bound}
\frac{||\mathcal{A}||_{\infty}^{\frac{1}{m-1}}}{||(-q)_+||_{\infty}^{\frac{1}{m-1}}}
    \leq \frac{1}{||z||_{\infty}} \leq 
    \frac{\alpha(F_{\mathcal{A}})}{||(-q)_+||_{\infty}^{\frac{1}{m-1}}}.
\end{equation}
From Theorem \ref{theorem main result} we have inequation \ref{equation of main result}.
Now combining inequations (\ref{equation of main result}) and (\ref{equation for inverse of solution bound}) we obtain the desired result.
\end{proof}

Here we consider a positive diagonal tensor of even order which is eventually a $P$-tensor, and find a property of the quantity $\alpha(F_{\mathcal{A}}).$

\begin{lemma}\label{lemma FA}
Let $\mathcal{A}\in T_{m,n}$ be a positive diagonal tensor, where $m$ is even. Then $\alpha(F_{\mathcal{A}})= \min_{i\in [n]}\{(a_{ii \cdots i})^{\frac{1}{m-1}} \}$.
Where $a_{ii \dots i}$ denote the main diagonal elements of $\mathcal{A}$, for $i=1, 2, ... ,n $.
\end{lemma}
\begin{proof}
Let $x\in \mathbb{R}^n$ be an arbitrary vector satisfied with $||x||_{\infty}=1.$ Then we have 
\begin{align*}
    \max_{i\in [n]} \{ x_i (F_{\mathcal{A}(x)})_i \} & = \max_{i\in [n]} \{ x_i (\mathcal{A}x^{m-1})^{[\frac{1}{m-1}]}_i \} \\
     & = \max_{||x||_{\infty} = 1} \{ x_i (a_{ii \dots i}x_i^{m-1})^{\frac{1}{m-1}}, \; i=1,2,\dots n \}\\
     & = \max_{||x||_{\infty} = 1} \{ (a_{ii \cdots i})^{\frac{1}{m-1}} x_i^2 , \; i=1,2,\dots n \}\\
     & \geq \min_{i\in [n]} \{ (a_{ii \cdots i})^{\frac{1}{m-1}} \} \; >0 .
\end{align*}
Thus 
\begin{equation}\label{geq FA}
\alpha(F_{\mathcal{A}}) = \min_{||x||_{\infty} = 1} \max_{i\in [n] } \{ x_i (F_{\mathcal{A}} (x))_i \} \geq \min_{i\in [n]} \{ (a_{ii \cdots i})^{\frac{1}{m-1}} \} .
\end{equation}
On the other hand by Theorem 4.2 of \cite{song2015error} we have
\begin{equation}\label{leq FA}
   \alpha(F_{\mathcal{A}}) \leq  \min_{i\in [n]} \{ (a_{ii \cdots i})^{\frac{1}{m-1}} \} .
\end{equation}
Thus from (\ref{geq FA}) and (\ref{leq FA}) we have, 
\begin{equation}
    \alpha(F_{\mathcal{A}}) = \min_{i\in [n]} \{ (a_{ii \cdots i})^{\frac{1}{m-1}} \} .
\end{equation}
\end{proof}

Here we present solution bound, absolute and relative error bound of the solution of TCP$(q, \mathcal{A}),$ where $\mathcal{A}$ is an even order positive diagonal tensor.

\begin{theorem}\label{corollary for error bound with positive diagonal tensor}
Let $\mathcal{A}\in T_{m,n}$ be a positive diagonal tensor and $m$ be even. Let $z\in \mathbb{R}^n$ be a solution of the TCP$(q,\mathcal{A})$ and $u$ be an arbitrary vector in $\mathbb{R}^n$. Then, 
\begin{align*}
    \frac{||v_1||_\infty (1 + (\max_{i\in [n]} \{ a_{ii \cdots i}\})^{\frac{1}{m-1}} ) - \sqrt{D_1}}{2\min_{i\in [n]} \{ (a_{ii \cdots i})^{\frac{1}{m-1}} \}} 
    & \leq ||z-u||_\infty \\
    & \leq \frac{||v_1||_\infty (1 + (\max_{i\in [n]} \{ a_{ii \cdots i}\})^{\frac{1}{m-1}} ) + \sqrt{D_1}}{2\min_{i\in [n]} \{ (a_{ii \cdots i})^{\frac{1}{m-1}} \}}
\end{align*}
and 
\begin{equation*}
   \frac{||(-q)_+||_{\infty}^{\frac{1}{m-1}}}{(\max_{i\in [n]} \{ a_{ii \cdots i}\})^{\frac{1}{m-1}}} 
    \leq ||z||_{\infty} \leq 
    \frac{||(-q)_+||_{\infty}^{\frac{1}{m-1}}}{\min_{i\in [n]} \{ (a_{ii \cdots i})^{\frac{1}{m-1}}\}},
\end{equation*}
where $v_1 =u -\max \{ 0, u-[(\mathcal{A}(u-z)^{m-1})^{[\frac{1}{m-1}]} + (\mathcal{A}z^{m-1} + q)^{[\frac{1}{m-1}]} ]\}$,
$D_1 = ||v_1||_{\infty}^2(1+ (\max_{i\in [n]} \{ a_{ii \cdots i}\})^{\frac{1}{m-1}})^2 -4(\min_{i\in [n]} \{ (a_{ii \cdots i})^{\frac{1}{m-1}}) v_t^2  \geq 0$ and $t$ satisfies $\max_{i\in [n]} \{(u-z)_i (\mathcal{A}(u-z)^{m-1})_i \} = (u-z)_t (\mathcal{A}(u-z)^{m-1})_t \geq 0 $ and $v_t \neq 0 $.
\end{theorem}
\begin{proof}
$\mathcal{A}$ is positive diagonal tensor of even order then $ a_{ii \cdots i}>0, \forall i\in [n]$ and all other entries of $\mathcal{A}$ are zeros. Then $\mathcal{A}$ is a $P$-tensor and $||\mathcal{A}||_{\infty} = \max_{i\in [n]} \{ a_{ii \cdots i} \}$ also from Lemma \ref{lemma FA}, $\alpha(F_{\mathcal{A}}) = \min_{i\in [n]} \{ (a_{ii \cdots i})^{\frac{1}{m-1}} \}.$
In Theorem \ref{for solution bound 2}, putting the values of $\alpha(F_{\mathcal{A}})$ and $||\mathcal{A}||_{\infty}$ in (\ref{equation for solution bound}) we obtain
\begin{equation}\label{equation for solution bound with positive diag}
   \frac{||(-q)_+||_{\infty}^{\frac{1}{m-1}}}{(\max_{i\in [n]} \{ a_{ii \cdots i}\})^{\frac{1}{m-1}}} 
    \leq ||z||_{\infty} \leq 
    \frac{||(-q)_+||_{\infty}^{\frac{1}{m-1}}}{\min_{i\in [n]} \{ (a_{ii \cdots i})^{\frac{1}{m-1}}\}}.
\end{equation}
In Theorem \ref{theorem main result}, putting the values of $\alpha(F_{\mathcal{A}})$ and $||\mathcal{A}||_{\infty}$ in (\ref{equation of main result}) we obtain
\begin{align}\label{equation for error bound with positive diag}
    \frac{||v_1||_\infty (1 + (\max_{i\in [n]} \{ a_{ii \cdots i}\})^{\frac{1}{m-1}} ) - \sqrt{D_1}}{2\min_{i\in [n]} \{ (a_{ii \cdots i})^{\frac{1}{m-1}} \}} 
    &\leq ||z-u||_\infty \\
    &\leq   
    \frac{||v_1||_\infty (1 + (\max_{i\in [n]} \{ a_{ii \cdots i}\})^{\frac{1}{m-1}} ) + \sqrt{D_1}}{2\min_{i\in [n]} \{ (a_{ii \cdots i})^{\frac{1}{m-1}}\}} ,
\end{align}
where $v_1=u -\max \{ 0, u-[(\mathcal{A}(u-z)^{m-1})^{[\frac{1}{m-1}]} + (\mathcal{A}z^{m-1} + q)^{[\frac{1}{m-1}]} ]\}$, 
$D_1 = ||v_1||_{\infty}^2(1+ (\max_{i\in [n]} \{ a_{ii \cdots i}\})^{\frac{1}{m-1}})^2 -4(\min_{i\in [n]} \{ (a_{ii \cdots i})^{\frac{1}{m-1}}) v_t^2  \geq 0$ and $t$ satisfies $\max_{i\in [n]} \{(u-z)_i (\mathcal{A}(u-z)^{m-1})_i \} = (u-z)_t (\mathcal{A}(u-z)^{m-1})_t \geq 0 $ and $v_t \neq 0 $.
\end{proof}

\begin{theorem}
Let $\mathcal{A} \in T_{m,n} $ be a positive diagonal tensor of even order, $0\neq z\in \mathbb{R}^n$ be a solution of TCP$(q,\mathcal{A})$ and $u$ be an arbitrary vector in $\mathbb{R}^n$. If $(-q)_+ \neq 0$ then
\begin{align*}
    & \frac{||v_1||_\infty (1 + (\max_{i\in [n]} \{ a_{ii \cdots i}\})^{\frac{1}{m-1}} ) - \sqrt{D_1}}{2 ||(-q)_+||_{\infty}^{\frac{1}{m-1}}}\\
    &\leq \frac{||z-u||_\infty}{||z||_\infty}\\ 
    &\leq 
        \frac{(\max_{i\in [n]} \{ a_{ii \cdots i}\})^{\frac{1}{m-1}}(||v_1||_\infty (1 + (\max_{i\in [n]} \{ a_{ii \cdots i}\})^{\frac{1}{m-1}} ) + \sqrt{D_1})}{2\min_{i\in [n]} \{ (a_{ii \cdots i})^{\frac{1}{m-1}}\} ||(-q)_+||_{\infty}^{\frac{1}{m-1}}}
\end{align*}
where $v_1=u -\max \{ 0, u-[(\mathcal{A}(u-z)^{m-1})^{[\frac{1}{m-1}]} + (\mathcal{A}z^{m-1} + q)^{[\frac{1}{m-1}]} ]\}$,
$D_1 = ||v_1||_{\infty}^2(1+ (\max_{i\in [n]} \{ a_{ii \cdots i}\})^{\frac{1}{m-1}})^2 -4(\min_{i\in [n]} \{ (a_{ii \cdots i})^{\frac{1}{m-1}}) v_t^2  \geq 0$ and $t$ satisfies  $\max_{i\in [n]} \{(u-z)_i (\mathcal{A}(u-z)^{m-1})_i \} = (u-z)_t (\mathcal{A}(u-z)^{m-1})_t \geq 0 $ and $v_t \neq 0 $.
\end{theorem}
\begin{proof}
Let $\mathcal{A}\in T_{m,n}$ be a positive diagonal tensor and $z\in \mathbb{R}^n$ be a solution of TCP$(q,\mathcal{A})$ and let $u$ be an arbitrary vector in $\mathbb{R}^n.$ Then by Theorem \ref{corollary for error bound with positive diagonal tensor}, the inequations (\ref{equation for solution bound with positive diag}) and (\ref{equation for error bound with positive diag}) hold. Since $(-q)_+ \neq 0$ and $z\neq 0$ from (\ref{equation for solution bound with positive diag}) we have 
\begin{equation}\label{equation for inverse of solution bound with positive diag}
       \frac{(\max_{i\in [n]} \{ a_{ii \cdots i}\})^{\frac{1}{m-1}}} {||(-q)_+||_{\infty}^{\frac{1}{m-1}}}
    \leq \frac{1}{||z||_{\infty}} \leq 
    \frac{\min_{i\in [n]} \{ (a_{ii \cdots i})^{\frac{1}{m-1}}\}}{||(-q)_+||_{\infty}^{\frac{1}{m-1}}}.
\end{equation}
Now combining inequations (\ref{equation for error bound with positive diag}) and (\ref{equation for inverse of solution bound with positive diag}) we obtain
\begin{align}\label{equation for relative error with positive diag}
    &\frac{||v_1||_\infty (1 + (\max_{i\in [n]} \{ a_{ii \cdots i}\})^{\frac{1}{m-1}} ) - \sqrt{D_1}}{2 ||(-q)_+||_{\infty}^{\frac{1}{m-1}}}\\
    &\leq \frac{||z-u||_\infty}{||z||_\infty}\\
    &\leq 
        \frac{(\max_{i\in [n]} \{ a_{ii \cdots i}\})^{\frac{1}{m-1}}(||v_1||_\infty (1 + (\max_{i\in [n]} \{ a_{ii \cdots i}\})^{\frac{1}{m-1}} ) + \sqrt{D_1})}{2\min_{i\in [n]} \{ (a_{ii \cdots i})^{\frac{1}{m-1}}\} ||(-q)_+||_{\infty}^{\frac{1}{m-1}}}.
\end{align}
\end{proof}

Here we give an numerical example to illustrate the result.
\subsubsection{Numerical Example}
Consider the tensor $\mathcal{A}\in T_{4,2}$ such that $a_{1111}=1, \; a_{2222}=8.$ For $z=\left( \begin{array}{c}
     z_1 \\
     z_2
\end{array} \right) \in \mathbb{R}^2,$ $\mathcal{A}z^2= \left( \begin{array}{c}
     z_1^3 \\
     8z_2^3
\end{array} \right) .$ Then $\mathcal{A}$ is a $P$-tensor. Let $q=\left( \begin{array}{c}
    1 \\
   -1
\end{array} \right) .$ Then the TCP$(q,\mathcal{A})$ is to find $z\in \mathbb{R}^2$ such that 
\begin{equation}\label{equation of example}
    z\geq 0, \;\; \mathcal{A}z^3 +q \geq 0, \;\; z^T (\mathcal{A}z^3 +q)=0.
\end{equation}
Solving (\ref{equation of example}) we have $z=\left( \begin{array}{c}
     0 \\
     .5
\end{array} \right).$ Then $\mathcal{A}z^3=\left( \begin{array}{c}
     0 \\
     1
\end{array} \right),$ $\mathcal{A}z^3 +q=\left( \begin{array}{c}
     1 \\
     0
\end{array} \right)$ and $(\mathcal{A}z^3 +q)^{\frac{1}{3}}=\left( \begin{array}{c}
     1 \\
     0
\end{array} \right).$ Now $||\mathcal{A}||_{\infty}=\max \{ 1,8 \}=8$, therefore $||\mathcal{A}||_{\infty}^{\frac{1}{3}}=2.$ Consider $u\in \mathbb{R}^n$ such that $u=\left( \begin{array}{c}
    0.5 \\
    0.4
\end{array} \right).$ Then $u-z=\left( \begin{array}{c}
     0.5 \\
     0.4
\end{array} \right),$ $\mathcal{A}(u-z)^3=\left( \begin{array}{c}
     0.125 \\
     0.008
\end{array} \right)$ and $(\mathcal{A}(u-z)^3)^{\frac{1}{3}}=\left( \begin{array}{c}
     0.5 \\
     0.8
\end{array} \right).$ Then $v=u -\max \{ 0, (\mathcal{A}(u-z)^{3})^{[\frac{1}{3}]} + (\mathcal{A}z^{3} + q)^{[\frac{1}{3}]} -u \}=$ 
$\left( \begin{array}{c}
     0.5 \\
     0.4
\end{array} \right)$ and $||v||_{\infty}=\max \{ 0.5, 0.4 \}=0.5.$ Now 
\begin{equation}
    \max_{i\in [2]} \{ (u-z)_i (\mathcal{A}(u-z)^3)_i\}= \{0.625, \; 0.0008\}=0.625,
\end{equation}
and the maximum occurs at $i=1$. Therefore $t =1$ and $v_t=v_1=0.5\;(\neq 0).$
Also
\begin{equation*}
    F_{\mathcal{A}}(z)= (\mathcal{A}z)^{\frac{1}{3}} = \left( \begin{array}{c}
     z_1 \\
    2 z_2
\end{array} \right).
\end{equation*}
Therefore $\alpha(F_{\mathcal{A}})= \min \{ 1,2 \}=1.$
Now
\begin{align*}
    D &= (0.5)^2 \cdot (1+2)^2 - 4 \cdot 1 \cdot (0.5)\\
    &= 0.25 >0 .
\end{align*}
Then upper bound of the error is $UB_{AKD} =\frac{0.5 \cdot (1+2) + \sqrt{0.25}}{2 \cdot 1}=1$ and the lower bound is $LB_{AKD} = \frac{0.5 \cdot (1+2)-\sqrt{0.25}}{2\cdot 1}= 0.5.$
But according to Theorem \ref{first proof of error bound} proposed by Zheng et al. in \cite{zheng2019global} the upper bound is $UB =\frac{(1+2)\cdot 2}{0.5}=1.5$ and the lower bound is $LB= \frac{1}{1+2} \cdot (0.5)=0.1667.$
Thus we see that $UB_{AKD} < UB$ and $LB < LB_{AKD}.$ Hence our error bound is improved.

\NI Obviously when $D \mapsto 0 $ the error bounds get better. Now we show that the error bound in Theorem \ref{theorem main result} is better than that of in Theorem \ref{first proof of error bound}.

\NI Here we are showing that our proposed error bound is sharper that the error bound proposed by Zheng et al. \cite{zheng2019global}.

\subsection{Comparison of upper bounds:}
The upper bound of error from inequation (\ref{first equation of error bound}) is $UB= ||\tilde{v}||_\infty  \cdot \frac{1 + ||\mathcal{A}||_{\infty}^{\frac{1}{m-1}}}{\alpha (F_{\mathcal{A}})}$ and the upper bound of error from inequation (\ref{equation of main result}) is $UB_{AKD}=\frac{||v||_\infty (1 + ||\mathcal{A}||_{\infty}^{\frac{1}{m-1}}) + \sqrt{D}}{2\alpha (F_{\mathcal{A}})}.$
Now we have $D = ||v||_{\infty}^2(1+ ||\mathcal{A}||_{\infty}^{[\frac{1}{m-1}]} )^2 -4\alpha(F_{\mathcal{A}}) v_t^2  \geq 0$ i.e., $4\alpha(F_{\mathcal{A}}) v_t^2 = ||v||_{\infty}^2(1+ ||\mathcal{A}||_{\infty}^{[\frac{1}{m-1}]} )^2 - D \geq 0,$ since $\alpha(F_{\mathcal{A}})\geq 0.$ Therefore 
\begin{equation}\label{equation for alg exp}
    (1+ ||\mathcal{A}||_{\infty}^{[\frac{1}{m-1}]} ) \geq \sqrt{\frac{D}{||v||_{\infty}^2}}.
\end{equation}
Note that by construction we have $||v|| \leq ||\tilde{v}||.$
By taking the ratio of the upper bounds we obtain,
\begin{align*}
   \frac{UB_{AKD}}{UB}= 
   \frac{\frac{||v||_\infty (1 + ||\mathcal{A}||_{\infty}^{\frac{1}{m-1}}) + \sqrt{D}}{2\alpha (F_{\mathcal{A}})}}{ ||\tilde{v}||_\infty  \cdot \frac{1 + ||\mathcal{A}||_{\infty}^{\frac{1}{m-1}}}{\alpha (F_{\mathcal{A}})} } \leq &\frac{(1 + ||\mathcal{A}||_{\infty}^{\frac{1}{m-1}}) + \sqrt{\frac{D}{||v||_{\infty}^2}}}{2 (1 + ||\mathcal{A}||_{\infty}^{\frac{1}{m-1}})}\\
    \leq & \frac{(1 + ||\mathcal{A}||_{\infty}^{\frac{1}{m-1}}) + (1 + ||\mathcal{A}||_{\infty}^{\frac{1}{m-1}})}{2 (1 + ||\mathcal{A}||_{\infty}^{\frac{1}{m-1}})} =1 \mbox{  by (\ref{equation for alg exp})}.
\end{align*}

\section{Conclusion}
In this article we introduce an error bound for the solution of tensor complementarity problem TCP$(q, \mathcal{A})$ given that $\mathcal{A}$ is a $P$-tensor and $q$ is a real vector. We establish a relative error bound for TCP with $P$-tensor. We find the value of $\alpha(F_{\mathcal{A}})$ for even order positive diagonal tensor. We establish absolute and relative error bound for TCP$(q, \mathcal{A})$ where $\mathcal{A}$ is an even order positive diagonal tensor. We prove that our proposed upper bound is sharper than the earlier upper bound of absolute error available in the literature. One numerical example is illustrated to support our result.

\section{Acknowledgment}
The author R. Deb is thankful to the Council of Scientific $\&$ Industrial Research (CSIR), India, Junior Research Fellowship scheme for financial support.
The author A. Dutta is thankful to the Department of Science and technology, Govt. of India, INSPIRE Fellowship Scheme for financial support.

\bibliographystyle{plain}
\bibliography{referencesall}

\end{document}